\documentclass[letter,11pt]{amsart}
\usepackage{amssymb} 
\usepackage{graphicx}
\usepackage{color}
\usepackage[font=footnotesize]{caption}
\usepackage{mathtools}
\usepackage{tikz}
\usepackage{hyperref}

\usepackage{float}
\usepackage{calc}
\def\thetitle{{An algorithm to compute the Teichm\"uller polynomial from matrices}}
\hypersetup{
   pdftitle=   \thetitle,
   pdfauthor=  {Hyungryul Baik and Chenxi Wu}
}

\newtheorem{thm}{Theorem}
\newtheorem*{thm*}{Theorem}

\newtheorem{prop}[thm]{Proposition}
\newtheorem{que}{Question}

\theoremstyle{remark}

\newtheorem*{rem}{Remark}
\usepackage{tikz}
\usetikzlibrary{calc,decorations.markings}
\usetikzlibrary{shapes,snakes}

\theoremstyle{definition}
\newtheorem*{defn*}{Definition}

\newcommand\C{\mathbb{C}}
\newcommand\R{\mathbb{R}}
\newcommand\Z{\mathbb{Z}}

\begin{document}

\title\thetitle

\author{Hyungryul Baik, Chenxi Wu \\ Appendix by KyeongRo Kim and
  TaeHyouk Jo}

\date{\today}

\begin{abstract}
In their precedent work, the authors constructed closed oriented hyperbolic
surfaces with pseudo-Anosov homeomorphisms from certain class of
integral matrices. In this paper, we present a very simple algorithm 
to compute the Teichm\"uller polynomial corresponding to those surface
homeomorphisms by first constructing an invariant track whose first homology group can 
be naturally identified with the first homology group of the surface, and 
computing its Alexander polynomial.  
 \end{abstract}

\maketitle

\section{Introduction} \label{sec:intro}

For orientable 3-manifolds, Thurston defined a norm, so-called \emph{Thurston norm}, on the second homology group. In this paper, by Thurston norm, we always mean 
the dual norm on the first cohomology group. The unit norm ball for Thurston norm is a rational polytope. Thurston showed that for a top-dimensional face $F$ of the unit norm ball, 
if an integral point in the cone $F \cdot \R_+$ is represented by a fibration of the 3-manifold over $S^1$, then all integral points in the same cone are also represented by fibrations. 
In this case, $F$ is called a \emph{fibered face}, and $F \cdot \R_+$ is called a \emph{fibered cone}. 

In \cite{McM00}, McMullen defined a polynomial invariant for fibered
faces, so-called \emph{Teichm\"uller polynomial}. 
If we are given a surface $S$ with a pseudo-Anosov monodromy $\phi$, 
we will simply say the Teichm\"uller polynomial for the pair $(S, \phi)$ to denote the Teichm\"uller polynomial for the fibered cone which contains a cohomology corresponding to the 
fibration defined by $S, \phi$. 

There have been a lot of interesting applications of Teichm\"uller
polynomials 
by many authors which we do not even attempt to make a complete list
here. 
One of the most noteworthy application of Teichm\"uller polynomial lies in
the study of minimal dilatation of pseudo-Anosov surface
homeomorphisms. See, for instance, \cite{FLM11, Hiro10, KKT13,
  Hiro14, Sun15}. 

There are also other polynomial invariants developed in the literature
as analogues to Teichm\"uller polynomial. For instance, Dowdall,
Kapovich, and Leininger defined the poynomial invariant, which they
call a McMullen polynomial, for the cohomology
classes of free-by-cyclic groups in \cite{DKL16} based on
\cite{DKL15}, which is in a sense a generalization of the Teichm\"uller polynomial. 
A similar result was obtained by Algom-Kfir,
Hironaka, and Rafi in \cite{AKHR15}. 
We also note that in \cite{McM15},
McMullen used so-called clique polynomials to give a sharp lower bound
on the spectral radius of a reciprocal Perron-Frobenius matrix with a
given size.

The first step in McMullen's construction of Teichm\"uller polynomial is
to define a module from the 2-dimensional lamination (i.e. the suspension of the stable lamination of the pseudo-Anosov monodromy of 
a fibration). This 2-dimensional lamination (up to isotopy) does not depend on the choice of the
fibration of the manifold. One can replace the stable lamination by
the invariant train track, and then the 2-dimensional lamination is
replaced by a branched surface. 


The definition of Teichm{\"u}ller polynomial in \cite{McM00} is
through an algorithm of computing them based on these invariant
train tracks. However, due to the importance of this invariant other
more efficient algorithms of computation in specific cases have been
developed, e.g. \cite{LV14}. 

In this paper, we present a simple algorithm to compute Teichm\"uller
polynomial for so-called \emph{odd-block surfaces} constructed from
some $\{0,1\}$-matrices in a precedent work of the authors
\cite{BRW16}. 
The definition and construction of such surfaces will be recalled in
Section \ref{sec:oddblocksurface}. 
Odd-block surfaces form a large class of translation surfaces which come with
pseudo-Anosov self-homeomorphisms. 
We observe that for this special case, 
the Teichm\"uller polynomial is just the Alexander
polynomial of an associated finitely presented group, and the latter
is fairly easy to understand and compute. 

This reduction of the computation of the Teichm{\"u}ller polynomial to
the computation of the Alexander polynomial is described in
Proposition \ref{prop:alex-teich}. As explained in the proof of
Proposition \ref{prop:alex-teich}, what makes our algorithm
to work for odd-block surface is that
 the matrix we start with in our construction indicates the transition matrix of a Markov partition of the surface under this
  pseudo-Anosov map, and the rectangular blocks of this Markov
  partition is ordered in a very regular way which allows us to
  analyze its lift to the abelian cover easily.
In particular, one gets a natural identification between the first
homology of the surface and the first homology of the train track we
construct in the algorithm (see Step I in Section \ref{sec:algorithm}),
and this allows us to compute the Teichm\"uller polynomial via the
Alexander polynomial of the branched surface which is the suspension of
the train track we deal with. 
Also, it is essential that the train track in our case is orientable
 as explained in the proof of Proposition \ref{prop:alex-teich}. 

After we recall the set up in Section \ref{sec:oddblocksurface}, we present our algorithm in
Section \ref{sec:algorithm}   using a running example, and a give a
proof of that our algorithm works in Section \ref{sec:proof}. 

\section{ Construction of odd-block surfaces} \label{sec:oddblocksurface}

In this section, we quickly recall the construction of surfaces given
in \cite{BRW16}. We start with an aperiodic non-singular $n \times n$
matrix $M$ with only entires $0$ and $1$ satisfying so-called
\emph{odd-block} condition. Namely, $M$ is said to be an odd-block
matrix if the following two conditions are satisfied: 
\begin{itemize}
\item[(i)] In each column of $A$, the non-zero entries form one consecutive block; and
\item[(ii)] There is a map $\phi:\{0,1,...,n\}\to \{0,1,...,n\}$ such that the entry $A_{ij}$ is odd if and only if $\min\{\phi(j-1),\phi(j)\} < i \le \max\{\phi(j-1),\phi(j)\}$.
\end{itemize}

In each column, inside the consecutive block of non-zero entries, there is a consecutive sub-block of
odd entries. Moreover, the "final position" of the odd-block in a column is the "initial position" of the odd-block in 
the next column where "final position" and "initial position" are just the values of the function $\phi$ in the definition. 
Hence the odd-blocks form a snake-shape. Note that the odd-block in
some column could be empty. The concept of odd-block matrices are
originally introduced in \cite{ThursEnt}, but the name was coined in
\cite{BRW16} where the properties of odd-block matrices were further
investigated.

Let us use the following running example: 
$$ M = \begin{pmatrix} 
0 & 0 & 0 & 0 & 0 & 0 & 1 & 1 \\
0 & 0 & 0 & 0 & 1 & 1 & 1 & 1 \\
0 & 0 & 0 & 0 & 1 & 1 & 1 & 0 \\
0 & 0 & 1 & 1 & 1 & 1 & 1 & 0 \\
0 & 0 & 1 & 1 & 1 & 0 & 0 & 0 \\
1 & 1 & 1 & 1 & 1 & 0 & 0 & 0 \\
1 & 1 & 1 & 0 & 0 & 0 & 0 & 0 \\
0 & 1 & 1 & 0 & 0 & 0 & 0 & 0 \\
\end{pmatrix} $$

In the rest of the section, we will follow the recipe given in
\cite{BRW16} to construct a surface of finite type and an
orientation-preserving pseudo-Anosov
homeomorphism from the above matrix $M$. In fact, one could construct
an orientation-reserving homeomorphism but since we are going to use
only orientation-preserving homeomorphisms for the rest of the paper, 
we only recall the process of constructing orientation-preserving
pseudo-Anosov homeomorphisms. 

The eigenvalue $\vec{v}$ of $M^T$ for the leading eigenvalue normalized so that
the $L^1$-norm is 1 is (approximately) 
$$ \{0.0740679, 0.0874795, 0.187546, 0.134026, 0.21327, 0.117861, 
0.139202, 0.0465469 \}.$$ Let $v_i$ denote the $i$th entry of
$\vec{v}$. Then one can use $\vec{v}$ to get a partition $\{x_0, x_1,
\ldots, x_8\}$ of $[0,1]$ so that $x_{i} - x_{i-1} = v_{i}$ for each
$i = 1, \ldots, 8$.  Let's consider the $8 \times 8$ grid diagram 
so that the boxes corresponds to the entries of $M$ flipped upside down. 
Identify each side of this grid diagram with the closed interval $[0,1]$ so that
that diagram represents the region $[0,1]\times [0,1]$. We adjusts the heights and widths
of rows and columns of this diagram so that the $i$th column has width $v_i$ and 
$(n-i)$th row has height $v_i$ for each $i$. 

In each box with entry 1, we draw a line segment connecting the top and the bottom of the box. 
By arranging them appropriately, one can always draw a graph of a continuous
piecewise-linear map $h_M$ as in Figure \ref{piecewiselinear}. As shown in
\cite{ThursEnt}, for the map $h_M$, $M$ is an extended incident
matrix, the set $\{x_0, \ldots, x_8\}$ is contained in the post-critical set of $h_M$, and the
leading eigenvalue of $M$ is the absolute value of the slope in each
piece. Say $\lambda_M$ is the leading eigenvalue.   

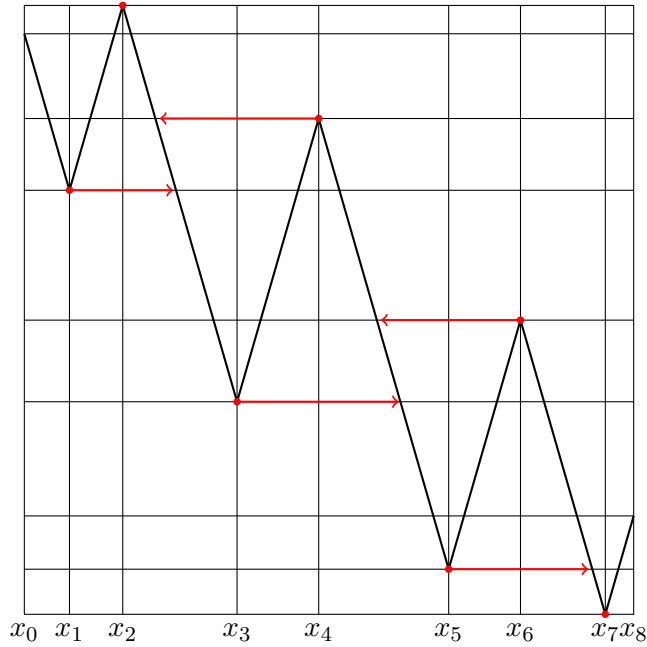
\begin{figure}[H]
  
  \begin{tikzpicture}[scale=0.6]
    \draw[-](0,0)--(13.5011,0)--(13.5011,13.5011)--(0,13.5011)--(0,0);
    \draw[-](1,0)--(1,13.5011);
    \draw[-](2.1811,0)--(2.1811,13.5011);
    \draw[-](4.7132,0)--(4.7132,13.5011);
    \draw[-](6.5227,0)--(6.5227,13.5011);
    \draw[-](9.4021,0)--(9.4021,13.5011);
    \draw[-](10.9933,0)--(10.9933,13.5011);
    \draw[-](12.8727,0)--(12.8727,13.5011);
   \draw[-](0,1)--(13.5011,1);
    \draw[-](0,2.1811)--(13.5011,2.1811);
    \draw[-](0,4.7132)--(13.5011,4.7132);
    \draw[-](0,6.5227)--(13.5011,6.5227);
    \draw[-](0,9.4021)--(13.5011,9.4021);
    \draw[-](0,10.9933)--(13.5011,10.9933);
    \draw[-](0,12.8727)--(13.5011,12.8727);
    \draw[-, thick](0,12.8727)--(1,9.4021)--(2.1811,13.5011)--(4.7132,4.7132)--(6.5227,10.9933)--(9.4021,1)--(10.9933,6.5227)--(12.8727,0)--(13.5011,2.1822);
    \draw[fill=red, red](1,9.4021) circle(2pt);
    \draw[->, red, thick](1,9.4021)--(3.3,9.4021);
    \draw[fill=red, red](2.1811,13.5011) circle(2pt);
    \draw[fill=red, red](4.7132,4.7132) circle(2pt);
    \draw[->, red, thick](4.7132,4.7132)--(8.3,4.7132);
    \draw[fill=red, red](6.5227,10.9933) circle(2pt);
    \draw[->,red, thick](6.5227,10.9933)--(3,10.9933);
    \draw[fill=red, red](9.4021,1) circle(2pt);
    \draw[->,red, thick](9.4021,1)--(12.5,1);
    \draw[fill=red, red](10.9933,6.5227) circle(2pt);
    \draw[->,red, thick](10.9933,6.5227)--(7.9,6.5227);
    \draw[fill=red, red](12.8727,0) circle(2pt);
    \node at (0,-0.4) {$x_0$};
    \node at (1,-0.4) {$x_1$};
    \node at (2.1822,-0.4) {$x_2$};
    \node at (4.7132,-0.4) {$x_3$};
    \node at (6.5227,-0.4) {$x_4$};
    \node at (9.4021,-0.4) {$x_5$};
    \node at (10.9933,-0.4) {$x_6$};
    \node at (12.8727,-0.4) {$x_7$};
    \node at (13.5011,-0.4) {$x_8$};
   \end{tikzpicture}
\caption{ The above figure shows the graph of a piecewise-linear map
  $h_M$ whose extended incident matrix is $M$. Here
  $\alpha(1)=\alpha(3)=\alpha(5)=1$, $\alpha(4)=\alpha(6)=-1$, and
  $\alpha(2),\alpha(7)$ are undefined.}
\label{fig:piecewiselinear}
\end{figure}

The red dots are $\{ (x_i, h_M(x_i) ) : i = 1,\ldots, 7\}$. One define so-called \emph{alignment function}
$\alpha: \{1, \ldots, 7\} \to \{-1, 1\}$ as follows; at each $i$, consider the horizontal line passing through 
$(x_i, h_M(x_i))$. If the horizontal line meets the graph of $h_M$ on the right side of $x_i$, then define 
$\alpha(i) = 1$, and if it meets the graph on the left side, define $\alpha(i) = -1$. If the horizontal line meets 
the graph of $h_M$ meets on both sides of $x_i$, then the construction fails. If the horizontal line does not meet
the graph at all, we leave $\alpha$ undefined at that point. In our running example, 
$\alpha(1) = \alpha(3) = \alpha(5) = 1$, $\alpha(4) = \alpha(6) = -1$, and $\alpha(2), \alpha(7)$ are undefined. 

At the moment, the alignment function $\alpha$ is defined only on a proper subset of $\{1, \ldots, 7\}$. 
As defined in \cite{BRW16}, we say $M$ satisfies the \textbf{alignment condition} if $\alpha$ can be extended to the 
entire $\{1, \ldots, 7\}$ while satisfying the following conditions; 
\begin{itemize}
\item[(a)] For critical $x_i, \,\,\alpha(i) = \begin{cases}
      \hfill -1 \hfill & \text{ if $x_i$ is a local max.} \\ 
      \hfill 1 \hfill & \text{ if $x_i$ is a local min.} \\
  		\end{cases}$
\item[(b)] For noncritical $x_i,\,\,\alpha(i) = \begin{cases}
      \hfill  \alpha(\phi_i) \hfill & \text{ if $h'(x_i)>0$.}\\
      \hfill - \alpha(\phi_i) \hfill & \text{ if $h'(x_i)<0$.}\\
  		\end{cases}$
\end{itemize}
In our example, the condition (b) is vacuously satisfied, since there is no non-critical $x_i$. 
Setting $\alpha(2) = -1$ and $\alpha(7) = 1$, we get an alignment function $\alpha$ defined at every $x_i$. 

Now let's go back to the grid diagram above but flip it along the horizontal axis once more so that the $i$th row has 
height $v_i$. We consider rectangles $R_i$, $i=1, \ldots, 8$, for the rows of this grid diagram. More precisely, 
$R_i$ is obtained from putting the boxes in $i$th row which are labeled with 1 side by side. In our example, all the 
1's in each row are consecutive, but it is not necessarily the case (see, for instance, Figure 3 of \cite{BRW16}). 
From the same grid diagram, we define another set of rectangles
$C_i$, $i=1, \ldots, 8$, for columns of the diagram. Namely, $C_i$ is
just the non-zero block in the $i$th column. 

We make a polygonal region $P_0$ by putting $R_{i+1}$ right below $R_i$ for each $i$ with the following rule; 
align $R_i, R_{i+1}$ on the left or right according to whether
$\alpha(i) = -1$ or $1$ respectively. Note that $P_0$ is the union of
$R_i$'s but at the same time the union of $C_i$'s. 

We define a piecewise-affine map $f$ from $P_o$ to $P_o$ as follows;
we map each $R_i$ to $C_i$ via an affine map which stretches $R_i$
vertically by the factor of $\lambda_M$ and compresses horizontally by
the factor of $1/\lambda_M$. When $h_M$ has negative derivative on
$[x_{i-1}, x_i]$, we compose this affine map with 180 degree
rotation. Then we obtain a piecewise-affine map on $P_0$ which is
well-defined in the interior of each $R_i$. For our example, see
Figure \ref{fig:piecewiseaffine}. 

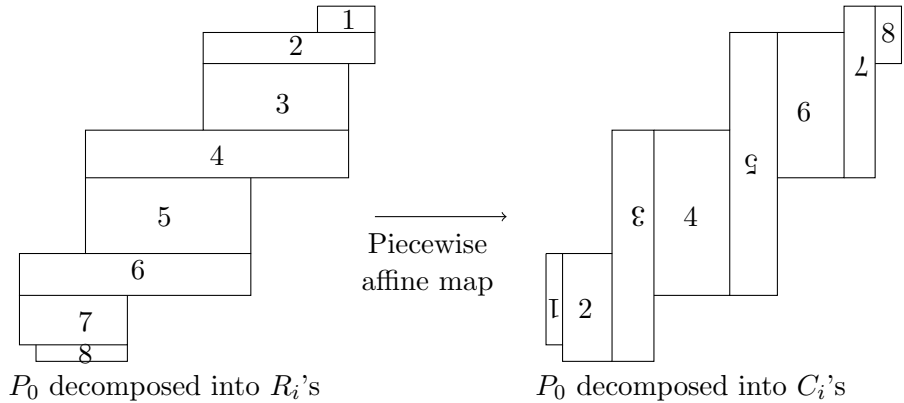
\begin{figure}[h]
  \begin{tikzpicture}[scale=0.35]
    \draw[-](20-13.5011,-9.4021)--(20-12.8727,-9.4021)--(20-12.8727,-12.8727)--(20-13.5011,-12.8727)--(20-13.5011,-9.4021);
    \draw[-](20-12.8727,-9.4021)--(20-10.9933,-9.4021)--(20-10.9933,-13.5011)--(20-12.8727,-13.5011)--(20-12.8727,-12.8727);
    \draw[-](20-10.9933,-9.4021)--(20-10.9933,-4.7132)--(20-9.4021,-4.7132)--(20-9.4021,-13.5011)--(20-10.9933,-13.5011);
    \draw[-](20-9.4021,-4.7132)--(20-6.5227,-4.7132)--(20-6.5227,-10.9933)--(20-9.4021,-10.9933);
    \draw[-](20-6.5227,-4.7132)--(20-6.5227,-1)--(20-4.7132,-1)--(20-4.7132,-10.9933)--(20-6.5227,-10.9933);
    \draw[-](20-4.7132,-1)--(20-2.1811,-1)--(20-2.1811,-6.5227)--(20-4.7132,-6.5227);
    \draw[-](20-2.1811,-1)--(20-2.1811,0)--(20-1,0)--(20-1,-6.5227)--(20-2.1811,-6.5227);
    \draw[-](20-1,0)--(20,0)--(20,-2.1811)--(20-1,-2.1811);
    \node at (20-8,-14.5) {$P_0$ decomposed into $C_i$'s};
    \draw[->](0,-8)--(5,-8);
    \node at (2,-9){Piecewise};
    \node at (2,-10.5){affine map};
    \draw[-](-2.1811,0)--(0,0)--(0,-1)--(-2.1811,-1)--(-2.1811,0);
    \draw[-](-2.1811,-1)--(-6.5227,-1)--(-6.5227,-2.1811)--(0,-2.1811)--(0,-1);
    \draw[-](-6.5227,-2.1811)--(-6.5227,-4.7132)--(-1,-4.7132)--(-1,-2.1811);
    \draw[-](-1,-4.7132)--(-1,-6.5227)--(-10.9933,-6.5227)--(-10.9933,-4.7132)--(-6.5227,-4.7132);
    \draw[-](-4.7132,-6.5227)--(-4.7132,-9.4021)--(-10.9933,-9.4021)--(-10.9933,-6.5227);
    \draw[-](-10.9933,-9.4021)--(-13.5011,-9.4021)--(-13.5011,-10.9933)--(-4.7132,-10.9933)--(-4.7132,-9.4021);
    \draw[-](-13.5011,-10.9933)--(-13.5011,-12.8727)--(-9.4021,-12.8727)--(-9.4021,-10.9933);
    \draw[-](-12.8727,-12.8727)--(-12.8727,-13.5011)--(-9.4021,-13.5011)--(-9.4021,-12.8727);
    
    \node at (-8,-14.5) {$P_0$ decomposed into $R_i$'s};
    \node at (-1,-0.5){1};
    \node at (-3,-1.5){2};
    \node at (16.5-20,-3.7){3};
    \node at (14-20,-5.7){4};
    \node at (12-20,-8){5};
    \node at (11-20,-10){6};
    \node at (9-20,-12){7};
    \node at (9-20,-13.2){8};
    \node [rotate=180] at (20-13.2,-11.5){1};
    \node at (20-12,-11.5){2};
    \node [rotate=180] at (20-10,-8){3};
    \node at (20-8,-8){4};
    \node [rotate=180] at (20-5.7,-6){5};
    \node at (20-3.7,-4){6};
    \node [rotate=180] at (20-1.5,-2.3){7};
    \node at (20-0.5,-1){8};
 \end{tikzpicture}
\caption{ A piecewise-affine map constructed from $h_M$. }
\label{fig: piecewiseaffine}
\end{figure}

Note that $f$ is not well-defined on line segments shared by
$R_i$ and $R_{i+1}$. For each such line segment, $f$ has two
images. To make $f$ well-defined, we need to identify those two
images. By doing this for every such line segment, one gets
well-defined $f$. On the other hand, this process still leaves a
problem of ill-definedness of $f^2$, so the images of the parts of
$P_0$ which are identified in the previous step need to be identified
again. Repeating this process, one gets infinitely many gluing
information on the boundary of $P_0$. By definition of $f$, these
gluing information are all on the horizontal edges of $P_0$. 

To get gluing information on the vertical edges of $P_0$, one just
repeat the same process with $f^{-1}$ instead of $f$, and focus on
the line segments shared by $C_i$'s. What the authors showed in
\cite{BRW16} is that the quotient of $P_0$ with these gluing
information on the boundary given by the map $f$ is a closed surface
equipped with a singular Euclidean metric with finitely many singular
points. By construction, $f$ is automatically a pseudo-Anosov map on
this surface.

Hence, we just obtained a translation surface with a pseudo-Anosov
homeomorphism on it. Let us call a translation surface obtained in
above process a \textbf{odd-block surface}.  

Even though one can produce a large class of examples using the
constructed described above, it seems to be not so easy to characterize
the odd-block surfaces with their intrinsic properties. We propose the
following open problem for future research. 

\begin{que} Find an interesting chracterization of odd-block surfaces. 
\end{que} 

In the rest of the paper, we will focus on the case when $n$ is even and any two consecutive entries of the eigenvector corresponding to the leading eigenvalue are different.
As shown in \cite{BRW16}  (very last part of Section 5), a careful
analysis of the gluing information on the boundary of $P_0$ reveals
that the resulting translation surface has $n+1$ cone points $Q_i$'s with cone
angle $\pi$ and a single cone point $Q$ with cone angle
$(n-1)\pi$. The assumption that two consecutive entries of the
eigenvector guarantees that the neighboring rectangles all have
different heights hence none of the $Q_i$ would disappear. 
Taking the double cover of the surface which are ramified
at $Q$ and $Q_i$'s, one gets a surface $S_g$ where the genus $g$ of
the surface is exactly $n/2$. In this case, the curves connecting
$Q_i$ and $Q_{i+1}$ in $P_0$ lift to loops which form a basis of
$H_1(S_g)$. It is also observed in \cite{BRW16} that with respect to
this basis, $M$ represents the action of the lift of the pseudo-Anosov
map constructed above on $H_1(S_g)$.

\section{An algorithm to compute Teichm\"uller polynomial for
  odd-block surfaces}  \label{sec:algorithm}

Let $n$ be an even number and $M$ be a $n \times n$ non-singular, aperiodic, odd-block
$\{0,1\}$-matrix, such that any two consecutive entries of the eigenvector corresponding to the leading eigenvalue are different. Say $S'$ is an odd-block surface and $\psi'$ is an
orientation-preserving pseudo-Anosov homeomorphism constructed as in
the previous section. 

As we remarked at the end of the last section, one can take a branched double cover $S$ of $S'$ whose
genus is exactly $n/2$, and $\psi'$ lifts to a pseudo-Anosov
homeomorphism $\psi$.  

Consider the mapping torus $N_\psi = S \times [0,1] / (x, 1) \sim
(\psi(x), 0)$. In $H^1(N_\psi, \mathbb{R})$, there exists a fibered
cone containing an integral cohomology class corresponding to the
fibration with fiber $S$ and monodromy $\psi$. 

In this case, we present an algorithm which computes the Teichm\"uller
polynomial associated with this fibered cone. 

As a running example, again we use $M$ defined in the previous
section; 
$$ M = \begin{pmatrix} 
0 & 0 & 0 & 0 & 0 & 0 & 1 & 1 \\
0 & 0 & 0 & 0 & 1 & 1 & 1 & 1 \\
0 & 0 & 0 & 0 & 1 & 1 & 1 & 0 \\
0 & 0 & 1 & 1 & 1 & 1 & 1 & 0 \\
0 & 0 & 1 & 1 & 1 & 0 & 0 & 0 \\
1 & 1 & 1 & 1 & 1 & 0 & 0 & 0 \\
1 & 1 & 1 & 0 & 0 & 0 & 0 & 0 \\
0 & 1 & 1 & 0 & 0 & 0 & 0 & 0 \\
\end{pmatrix} $$

\noindent Step I. Compute the eigenvectors $\vec v^j$ of $M$, and then for each non-zero entry $v_i^j$ of $\vec v^j$, write $\sum_j v_i^js_j$ as a superscript of every entry of $i$-th column of $M$. 

In our example, the only eigenvector $\vec v^t$ is $(1, 0, 0, 0, -1, 0, 1, 0)$. Hence the result of Step I for our example is 
$$\begin{pmatrix} 
0    & 0 & 0 & 0 & 0 & 0 & 1^s & 1 \\
0    & 0 & 0 & 0 & 1^{-s} & 1 & 1^s & 1 \\
0    & 0 & 0 & 0 & 1^{-s} & 1 & 1^s & 0 \\
0    & 0 & 1 & 1 & 1^{-s} & 1 & 1^s & 0 \\
0    & 0 & 1 & 1 & 1^{-s} & 0 & 0 & 0 \\
1^s & 1 & 1 & 1 & 1^{-s} & 0 & 0 & 0 \\
1^s & 1 & 1 & 0 & 0 & 0 & 0 & 0 \\
0    & 1 & 1 & 0 & 0 & 0 & 0 & 0 \\
\end{pmatrix} $$

This is an example where the first betti number of the mapping cylinder is 2. The case when the first betti number is 1 is trivial in that the Teichm\"uller polynomial, as well as the Alexander polynomial of the suspended train track (c.f. Proposition 1), would both be identical to the characteristic polynomial of the original matrix. Note that in this case the Steps II-IV would not change the matrix, hence this is consistent with the result of our algorithm.\\

\noindent Step II. We push all superscripts $\sum n^js_j$ to the right in each row. In this process, 
the entries of $M$ are multiplied by $t_j$ or $t^{-1}_j$ by applying the following rules repeatedly;
\begin{itemize}
\item ${}^{s_j} 1 \mapsto t_j^{s_j}$
\item ${}^{-s_j} 1 \mapsto {\left(t_j^{-1}\right)}^{-s_j}$
\end{itemize} 

The result of Step II for our example is 
$$\begin{pmatrix} 
0 & 0 & 0 & 0 & 0 & 0 & 1 & t^s \\
0 & 0 & 0 & 0 & 1 & 1/t & 1/t & 1 \\
0 & 0 & 0 & 0 & 1 & 1/t & 1/t & 0 \\
0 & 0 & 1 & 1 & 1 & 1/t & 1/t & 0 \\
0 & 0 & 1 & 1 & 1^{-s} & 0 & 0 & 0 \\
1 & t & t & t & t & 0 & 0 & 0 \\
1 & t & t^s & 0 & 0 & 0 & 0 & 0 \\
0 & 1 & 1 & 0 & 0 & 0 & 0 & 0 \\
\end{pmatrix} $$

\noindent Step III. Let $r$ be some row with superscript $\sum_j n^j s_j$ at the end. Now look at the rows above $r$ which corresponds to decreasing piece in the piecewise-linear map (see the previous section for details).  Say $r_1, \ldots, r_k$ are such rows where each row
$r_i$ has superscript $\sum_jn^j_i s_j$ at the end. Then replace the row $r$ by 
$$ r \cdot \prod_j t_j^{-\Sigma_i n^j_i} + \prod_j (t_j^{n^j} -1) \cdot \Sigma_{l=1}^k\left(\prod_j t_j^{-\Sigma_{i=1}^l n^j_i}\right) r_l .$$
and add a entry $\prod_jt_j^{n_j}-1$ to the end.

The final result for our example is 
$$\begin{pmatrix} 
0 & 0 & 0 & 0 & 0 & 0 & 1 & t & t-1\\
0 & 0 & 0 & 0 & 1 & 1/t & 1/t & 1 & 0\\
0 & 0 & 0 & 0 & 1 & 1/t & 1/t & 0 & 0\\
0 & 0 & 1 & 1 & 1 & 1/t & 1/t & 0 & 0\\
0 & 0 & 1/t & 1/t & 2/t-1 & 2/t(1/t-1) & 2/t(1/t-1) & 1/t-1 & 1/t-1\\
1 & t & t & t & t & 0 & 0 & 0 & 0\\
t & t^2 & t^2+t-1 & t^2-1 & t^2+t-2 & 2-2/t & 2-2/t & t-1 & t-1\\
0 & 1 & 1 & 0 & 0 & 0 & 0 & 0 & 0\\
\end{pmatrix} $$

\noindent Step IV. Let $M$ be the resulting matrix from Step III, and
let $N$ be the matrix of the same size as $M$ obtained by adding the
zero column to the right of the identity matrix. And take the greatest
common divisor of the largest minors of the matrix $M-xN$.

In our example, a direct computation shows that the gcd of the
largest minors is the following polynomial; 
$$  x^7 -x^6 - (2t+5+2/t)x^5 + x^4 - x^3 +(2t+5+2/t)x^2 + x -1.$$

We shall show in the next section, that:
\begin{thm}
The polynomial obtained from steps I-IV is the Teichm\"uller polynomial of the pseudo-Anosov map on $S$.
 \end{thm}

\begin{rem}
We can verify the correctness of our algorithm in this example by
using McMullen's algorithm. The ``thickening'' construction in
the previous section gives us an Markov partition of the flat surface into 16
rectangular regions. This induces an invariant train track by associating an edge with each rectangle and a vertex with the each connected component of the union of the left edges (c.f. the proof section later). Hence, by section 3 of \cite{McM00}, we have:
\begin{equation*}
\Theta(u,t)=
{\det\left(u\mathbf{I}-\left(\begin{array}{cccccccccccccccc} 
0&0&0&0&0&0&1&t&0&0&0&0&0&0&0&0\\
0&0&0&0&0&0&0&0&0&0&0&0&1&1&1&1\\
0&0&0&0&1&1/t&1/t&0&0&0&0&0&0&0&0&0\\
0&0&0&0&0&0&0&0&0&0&1&1&1&1&1&0\\
0&0&1&1&1&0&0&0&0&0&0&0&0&0&0&0\\
0&0&0&0&0&0&0&0&1&1&1&1&1&0&0&0\\
1&t&t&0&0&0&0&0&0&0&0&0&0&0&0&0\\
0&0&0&0&0&0&0&0&0&1&1&0&0&0&0&0\\
0&0&0&0&0&0&0&0&0&0&0&0&0&0&1&1\\
0&0&0&0&1&1/t&1/t&1&0&0&0&0&0&0&0&0\\
0&0&0&0&0&0&0&0&0&0&0&0&1&1&1&0\\
0&0&1&1&1&1/t&1/t&0&0&0&0&0&0&0&0&0\\
0&0&0&0&0&0&0&0&0&0&1&1&1&0&0&0\\
1&t&t&t&t&0&0&0&0&0&0&0&0&0&0&0\\
0&0&0&0&0&0&0&0&1&1&1&0&0&0&0&0\\
0&1&1&0&0&0&0&0&0&0&0&0&0&0&0&0
\end{array}\right)\right)\over \det\left(u\mathbf{I}-\left(
\begin{array}{ccccccccc}
0&0&0&0&0&0&1&0&0\\
0&0&0&0&0&0&0&0&1\\
0&0&0&0&1&0&0&0&0\\
0&0&0&0&0&0&0&1&0\\
0&0&1&0&0&0&0&0&0\\
0&0&0&0&0&1&0&0&0\\
1&0&0&0&0&0&0&0&0\\
0&0&0&1&0&0&0&0&0\\
0&1&0&0&0&0&0&0&0
\end{array}
\right)\right)}
\end{equation*}
\begin{equation*}
={x^8 - 2x^7 + (-2t - {2\over t} - 4)x^6 + (2t + {2\over t}+ 6)x^5 -
  2x^4 + (2t +{2\over t}+6)x^3 +(-2t - {2\over t} - 4)x^2 - 2x + 1
  \over x-1}
\end{equation*}
\begin{equation*}
=x^7 -x^6 - (2t+5+2/t)x^5 + x^4 - x^3 +(2t+5+2/t)x^2 + x -1.
\end{equation*}

\end{rem}

\begin{rem}
In the case when $n$ is odd, or more generally, when the genus of the surface constructed via the procedure \cite{BRW16} is of a genus smaller than $n/2$, the above procedure can still be used to calculate the Teichm{\"u}ller polynomial so long as Step I above is modified to only consider those eigenvectors corresponding to (absolute) cohomology classes of the resulting surface.
\end{rem}

\section{Proof of the main theorem} \label{sec:proof} 

\begin{proof}[Proof of theorem 1]

\noindent Step I.  We describe the invariant train track and the train
track map.\\

As we explained in Section \ref{sec:oddblocksurface}, an odd-block
surfaces is built by gluing blocks where each block corresponds to a
column of the given matrix, say $M$, and then we take a ramified double cover so that the 
invariant foliation is orientable.  
We call a block to be ``in the front'' if it is in the polygon $P_0$ described in Section \ref{sec:oddblocksurface}, ``in the back'' if otherwise. 

 We call the part of the surface consisting of the original blocks the ``front'' and the part consisting of the other blocks the ``back''. Let ($S_g, \psi$) be the pair of the surface an the psuedo-Anosov map we obtain. 

Now we can build an invariant (topological) train track as follows: 
each block corresponds to an edge and each connected component of the union of vertical edges of the blocks corresponds to a vertex. 
See Figure \ref{fig:traintrack}. We call an edge ``in the front'' if the corresponding block is in the front, ``in the back'' if the corresponding block is in the back.

\begin{figure}
\begin{tikzpicture}
\draw[-] (0,2)--(3,2);
\draw[dotted](0,2)--(0,0);
\draw[dotted](3,2)--(3,-1);
\draw[-](0,0)--(1.4,0)--(1.4,-1)--(3,-1);
\draw[red,thick] (1.9,-1) arc (0:90:0.5);
\draw[red,thick,dashed] (1.4,-0.5) arc (90:180:0.5);
\draw[red,thick,dashed] (0.9,0) arc (180:270:0.5);
\draw[red,thick] (1.4,-0.5) arc (270:360:0.5);
\draw[red,thick](0.9,0)arc(0:90:0.9);
\draw[red,thick](1.9,0)arc(180:90:1.1);

\draw[-] (4,-1)--(7,-1);
\draw[dotted](4,-1)--(4,1);
\draw[dotted](7,-1)--(7,2);
\draw[-](4,1)--(5.4,1)--(5.4,2)--(7,2);
\draw[red,thick] (5.9,1) arc (180:270:1.1);
\draw[red,thick,dashed] (5.4,1.5) arc (-90:-180:0.5);
\draw[red,thick,dashed] (5.4,1.5) arc (90:180:0.5);
\draw[red,thick] (5.4,1.5) arc (-90:0:0.5);
\draw[red,thick](5.4,1.5)arc(90:0:0.5);
\draw[red,thick](4.9,1)arc(0:-90:0.9);
\end{tikzpicture}

\caption{construction of a train track} 
\label{fig:traintrack}
\end{figure}
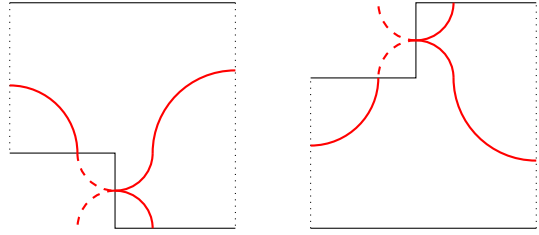

On this graph, call it $\Gamma$, we get a map, call it $\psi_\ast$, induced by the surface map $\psi$. \\

Choose the vertex corresponding to the left-edge of the left-most
rectangular block as the base point $v$. Then the invariant train track $\lambda$ is as in Figure \ref{fig:traintrack_full}

\begin{figure}
\begin{tikzpicture}[ ray/.style={decoration={markings,mark=at position .5 with {
      \arrow[>=latex]{>}}},postaction=decorate}
]
\node at (-0.1,1){$v$};
\draw[-](0,0.9)--(0,1.1);
\draw[ray](0,0.9)arc(-180:0:0.35);
\draw[ray](0,1.1)arc(180:0:0.35);
\draw[-](0.7,0.9)arc(180:0:0.1);
\draw[-](0.7,1.1)arc(-180:0:0.1);
\draw[ray](0.9,0.9)arc(-180:0:0.5);
\draw[ray](0.9,1.1)arc(180:0:0.5);
\node at (2.4,1){$\cdots$};
\draw[ray](3.1,0.9)arc(-180:0:0.45);
\draw[-](4,0.9)--(4,1.1);
\draw[ray](3.1,1.1)arc(180:0:0.45);
\node at (0.4,0.3){$\alpha_1$};
\node at (1.4,0.1){$\alpha_2$};
\node at (3.5,0.2){$\alpha_n$};
\node at (0.4,1.7){$\beta_1$};
\node at (1.4,1.9){$\beta_2$};
\node at (3.5,1.8){$\beta_n$};
\end{tikzpicture}

\caption{The invariant train track} 
\label{fig:traintrack_full}
\end{figure}
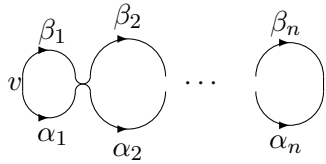

Here, $\alpha_i$ are the edges in the front, and $\beta_i$ are the edges in the back. The fundamental group of $\lambda$: $\pi_1(\lambda,v)$ is a free group generated by 
$$\gamma_i=\prod_{j\in\{1,2,\dots i-1\}}\beta_j\alpha_i\beta_i^{-1}(\prod_{j\in\{1,2,\dots i-1\}}\beta_j)^{-1}\ .$$

Here, if $S$ is an ordered set of indices, $\prod_{i\in S} g_i$ is the product of elements $g_i$ in the order determined by the order of $S$.\\

\noindent Step II. Next, we describe an isotopic train track map that fixes a base point.\\

Let the ``left-most'' vertex, $v$, of $\Gamma$ be the base point. For each loop $\gamma$ in $\Gamma$ based at $v$, $\psi_\ast(\gamma)$ is a loop conjugate to a loop based at $v$. This conjugacy can be written uniquely as follows:  for each loop, we conjugate it to a based loop through an embedded path consisting of such edges. \\

Let $\psi_1$ be the the map on $\Gamma$ which maps each based loop $\gamma$ to the conjugate based loop of $\psi_\ast(\gamma)$ as described above. Then this induces a homomorphism ${\psi_1}_\ast$ from $\pi_1(\Gamma, v)$ to itself. \\

We can now write down the map  ${\psi_1}_\ast$ under the generating set $\gamma_i$ using 1-dimensional PCF map $f$ we started with, whose incidence matrix is $M$.\\

Recall that the post-critical set decomposes the interval into $n$ segments, which we label $I_1,\dots I_n$. Let $\mathcal{I}$ be the index set of segments on which $f$ is increasing and $\mathcal{D}$ be the index set on which $f$ is decreasing, and let $w_i$ be the word (also seen as an ordered set) consisting of the indices of the edges in the image of $I_i$. Then, by the ``thickening'' construction in  Section 2, we know that the image of $\alpha_i$ under $\psi$ is $\prod_{j\in w_i}\alpha_j$ if $i\in \mathcal{I}$ and  $\prod_{j\in w_i}\beta_j$ if $i\in \mathcal{D}$, and the image of $\beta_i$ is s $\prod_{j\in w_i}\alpha_j$ if $i\in \mathcal{I}$ and  $\prod_{j\in w_i}\beta_j$ if $i\in \mathcal{D}$.\\

Hence, we have:

$${\psi_1}_\ast(\gamma_1)=\prod_{j\in w_1}\gamma_j$$
$${\psi_1}_\ast(\gamma_i)=\prod_{k\in \mathcal{D}\cap\{1,\dots i-1\}}(\prod_{j\in w_k}\gamma_j)\prod_{j\in w_i}\gamma_j(\prod_{k\in \mathcal{D}\cap\{1,\dots i-1\}}(\prod_{j\in w_k}\gamma_j))^{-1}$$

\noindent Step III. We now describe the infinite cover of the
train track graph as well as the map in II. In particular, we explain
what choices have been made in the previous steps. \\

Note that the matrix $M$ describes the induced action of $\psi$ on the first homology of $S_g$. For simplicity, suppose the dimension of the eigenspace corresponding to the eigenvalue 1 of $M$ is 1. Everything works exactly the same in the case the eigenspace has higher dimension. 

Let $\vec x$ be the eigenvector corresponding to the eigenvalue 1 of $M$. 
Now cut each edge of $\Gamma$ in the front near the right endpoint of such an edge. Consider the $\Z$-copies of this cut $\Gamma$, and enumerate as $\Gamma_i$ with $i \in \Z$. 
On each copy of $\Gamma$, the edge corresponds to the $i$-th block of the polygon $P_0$ in Section 2 is called $i$-th edge. 
Now glue $i$-th edge of $\Gamma_j$ to the $i$-th edge of $\Gamma_{j+x_i}$ for all $i$ and $j$ where $x_i$ denotes the $i$-th entry of the vector $\vec x$. 
Then we call the resulting infinite graph $\Gamma_\infty$ which is $\Z$-fold cover of $\lambda$. 
Let $\psi_\infty$ be the lift of ${\psi_1}_\ast$ to $\Gamma_\infty$ so that all the lifts of the based point $v$ are fixed.  \\

\noindent Step IV. Finally we obtain the algorithm from the map in III. \\ 

Let $G$ be the HNN extension of $\pi_1(\lambda)$ with respect to the map we obtained in the previous step, i.e.,  
$$G = <\pi_1(\lambda), u : u\gamma u^{-1} = {\psi_1}_\ast (\gamma), \forall \gamma \in \pi_1(\lambda) >. $$ 

We can now compute the Alexander polynomial of $G$ by Fox Calculus
\cite{F53} (also c.f. \cite{H97}) and the Alexander Matrix is of size
$(2g+1)\times 2g$ because the presentation above has $2g+1$ generators
($\gamma_1,\dots, \gamma_{2g}, u$) and $2g$ relations
$R_i={\psi_1}_\ast(\gamma_i)u\gamma_i^{-1}u^{-1}$, $i=1,2,\dots 2g$,
and is of the form $(M-[xI_{2g},0])^T$ where $M$ is the $2g\times (2g+1)$
matrix obtained by our algorithm and $x$ is the variable corresponding
to the generator $u$. \\

Lastly, we show that the Alexander polynomial of $G$ is the Teichm\"uller polynomial of $S$.

\begin{prop} \label{prop:alex-teich}
The Teichm\"uller polynomial for the pair $(S, \psi)$ of odd-block
surface and a pseudo-Anosov homeomorphism (obtained as in
Section 2) coincides with the Alexander polynomial for the corresponding mapping torus of $\mathcal{T}$. 
\end{prop} 
\begin{proof}
Let $A = H_1(M_\psi, \Z)/\mbox{torsion}$, and let $\hat{A} = Hom(A, \C^*)$. Also, let
$H = Hom(H^1(S, \Z)^\psi, \Z) \cong \Z^b$ be the dual of the $\psi$-invariant cohomology of $S$. 
The we have the splitting $A = H \oplus \Z$, and $(t, u)$ be the
coordinates on $\hat{A}$ adapted to the splitting. Note for our
construction of $\tau$ in Step I, there is an identification
between $H_1(S)$ and $H_1(\tau)$ (see also \cite[Theorem 4]{BRW16}). 
Hence $H$ is also the maximum
subgroup of $H_1(\tau)$ that is invariant under $\psi$. 
 
The rest of the proof is basically just copying and pasting the proof of Theorem 7.1 of
\cite{McM00}. 
Let $P(t)$ and $Q(t)$ denote the action of $\widetilde{\psi}$ on
$Z_1(\tau, \C_t)$ and $H_1(\tau, \C_t)$ respectively. Here the twisted
coefficient $\C_t$ takes value in $\C[\hat{H}]$ and corresponds to the
abelian covering defined by $\hat{H}$. 

Then by definition, the Alexander polynomial for $\mathcal{T}$ is $\det(uI - Q(t))$, and the Teichm\"uller polynomial for $(S, \psi)$
is $\det(uI - P(t))$. On the other hand, since $\tau$ is 1-dimensional
and orientable, we have $Z_1(\tau, \C_t) \cong H_1(\tau, \C_t)$ 
for any character $t \in \hat{H}$.  
Note that the orientability of $\tau$ is essential here to get this
identification between $Z_1(\tau, \C_t)$ and $H_1(\tau, \C_t)$. For
instance, by Corollary 2.4 of \cite{McM00}, when $\tau$ is
non-orientable, the rank of $Z_1(\tau, \C_t)$ is strictly smaller, so
it cannot be isomorphic to $H_1(\tau, \C_t)$. 

Note that in \cite{McM00},
$Q(t)$ denotes the action on $H_1(S, \C_t)$ not on $H_1(\tau,
\C_t)$, and the Alexander polynomial there means the Alexander
polynomial of the 3-manifold. In that case, the Alexander polynomial
is just a factor of Teichm\"uller polynomial. This is mainly due to the
fact that the map in (7.1) of \cite{McM00} is a mere surjection not an
isomorphism. On the other hand, in our case, $Q(t)$ is defined using
$H_1(\tau, \C_t)$ not $H_1(S, \C_t)$, and we use the Alexander
polynomial of the branched surface which is the suspension of the
train track. 
so the range of the map $\pi$ in  (7.1) of
\cite{McM00} becomes the same group as the domain, and we have the
isomorphism. This shows that $P(t) = Q(t)$, and the proposition follows. 
\end{proof}

This concludes the proof of the main theorem.
\end{proof}

\textbf{Remark.} In general, $H_1(S)$ is a quotient of $H_1(\tau,
\C_t)$, hence the Teichm\"uller polynomial is a specialization of the
Alexander polynomial for $\mathcal{T}$ when $\tau$ is orientable. In
our case, due to the identification between  $H_1(S)$ and $H_1(\tau)$,
this specialization is not needed. \\

\textbf{Remark.} In Step IV, say the eigenspace had dimension $k$, and we have the eigenvectors $\vec x_1, \ldots, \vec x_k$. 
Then we consider the $\Z^k$-copies of the cut $\Gamma$. So now each copy has $k$ coordinates, so write $\Gamma_{t_1, \ldots, t_k}$ with $(t_1, \ldots, t_k) \in \Z^k$. Then for each $i$-th edge of $\Gamma_{t_1, \ldots, t_k}$ is glued to the $i$-th edge of 
$\Gamma_{t_1 + (\vec x_1)_i, t_2 + (\vec x_2)_i, \ldots, t_k + (\vec x_k)_i}$ 
where $(\vec x_j)_l$ denotes the $l$-th entry of the vector $\vec x_j$. 
This defines a $\Z^k$-fold cover of $\Gamma$ we need to compute the Teichm\"uller polynomial.

\section{Acknowledgements}

We greatly appreciate Ahmad Rafiqi for many helpful discussions and
comments. In particular, we are indebted to Ahmad for the running
example in the paper whose Teichm\"uller polynomial was confirmed by
his computation. We also thank Erwan Lanneau for a lot of inspiring discussions.  
Finally we thank the anonymous referee for helpful comments which
greatly improved the readability of our paper. 

The first author was partially supported by the ERC Grant Nb. 10160104. \\[10pt]

\newpage
\appendix
\section{Odd-block matrices with entries bigger than 1}  \label{appendix}
\begin{center}
\author{KyeongRo Kim and TaeHyouk Jo}
\end{center}

In the main text of the paper, the authors provided an algorithm to
compute the Teichm\"uller polynomial for odd-block surfaces, and the
odd-block surfaces are constructed as described in Section
\ref{sec:oddblocksurface} following \cite{BRW16}. One of the
limitations of the construction of odd-block surfaces is that the given odd-block
matrix is assumed to have only 0, 1 entires. The main purpose of this
assumption is to guarantee the uniqueness of the corresponding
piecewise-linear map $h_M$. See Figure \ref{fig:nonunique} for an example of an odd-block matrix with entries bigger
than 1 where $h_M$ is not unique.
\begin{figure}

\includegraphics[width=0.4\textwidth]{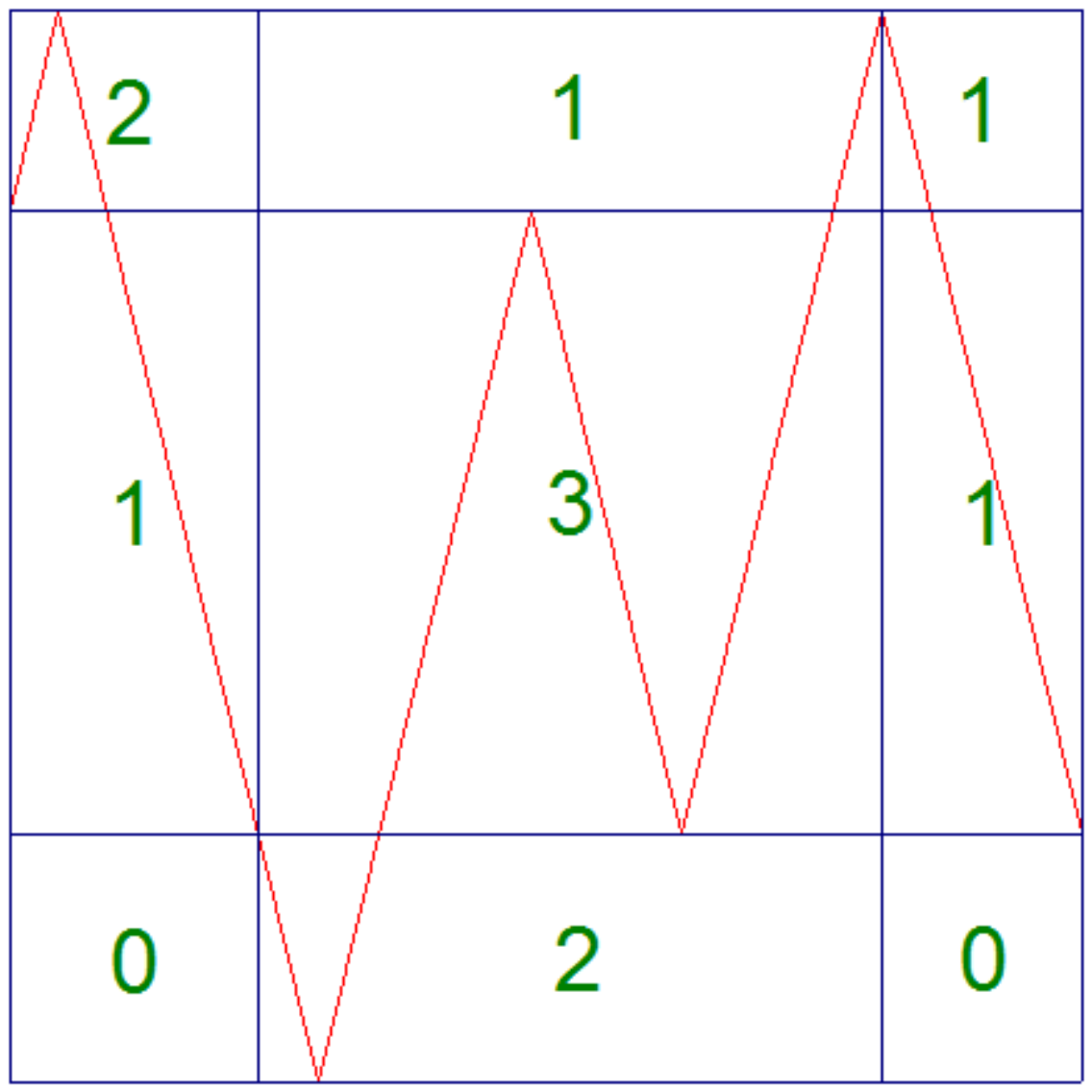}  \includegraphics[width=0.4\textwidth]{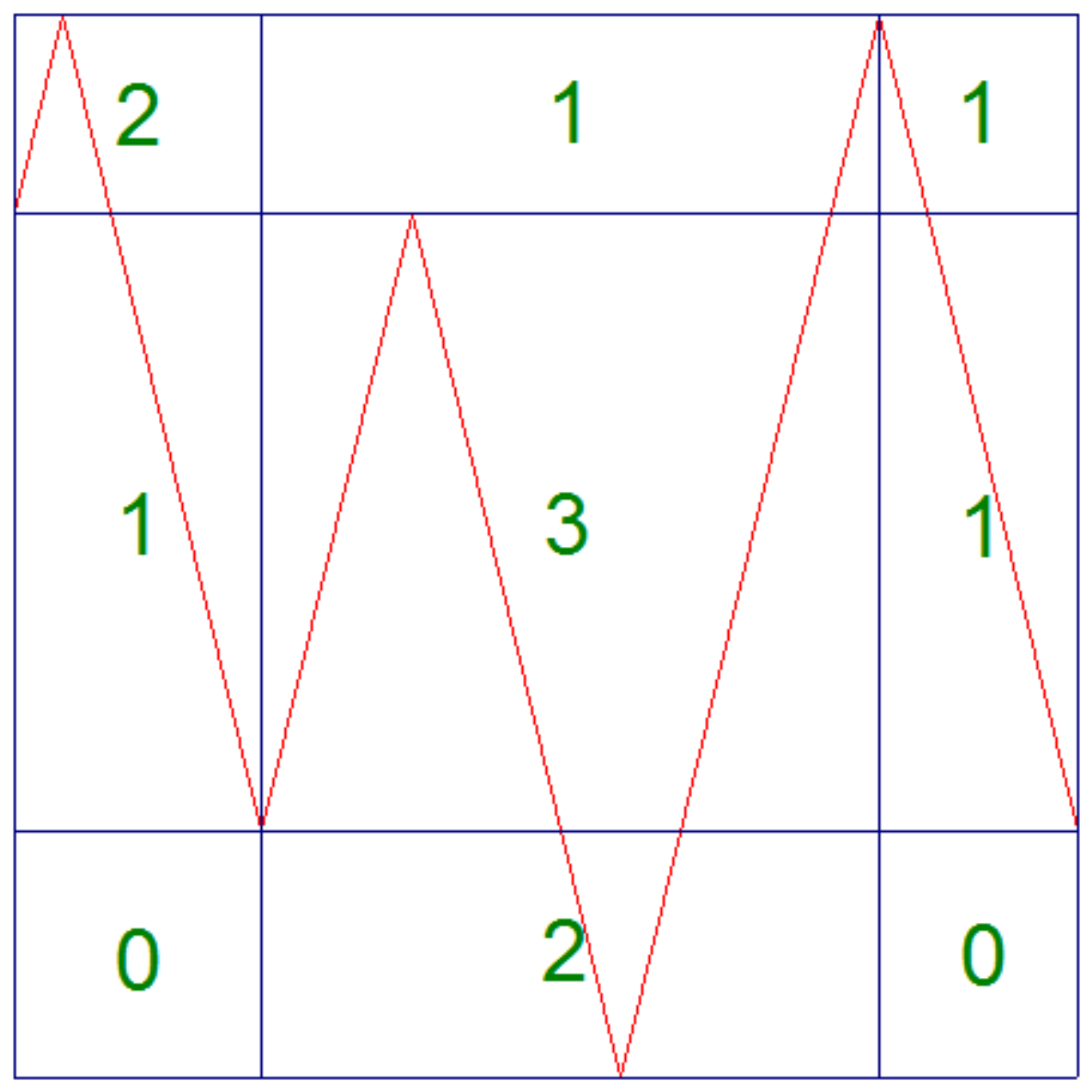}
\caption{The piecewise-linear map $h_M$ is not uniquely determined in
  the above example.} 
\label{fig:nonunique}
\end{figure}

On the other hand, for any given
odd-block matrix, one can obtain an odd-block matrix with
$\{0,1\}$-entries with essentially the same information. More
precisely, we prove the following theorem. 

\begin{thm}
\label{thm:appendix}
Let $M$ be an $n \times n$ non-singular, aperiodic, odd-block, nonnegative integral matrix. For each choice of the piecewise-linear map $h_M$, then there exists an aperiodic, odd-block matrix $N$ with only $\{0, 1\}$-entries such that $h_N$ coincides with $h_M$. Furthermore, the leading eigenvalue of $N$ is the leading eigenvalue of $M$.
\end{thm}

Let $M$ be a matrix as in Theorem \ref{thm:appendix}. For instance,
one can consider an example shown in the left part of Figure \ref{fig:generaloddblock}. We show that there is a canonical way to convert $M$ into another
aperiodic odd-block matrix $\overline{M}$ with only entries $0$ and
$1$, having the same leading eigenvalue (say $\lambda$). 

\begin{figure}

\includegraphics[width=0.4\textwidth]{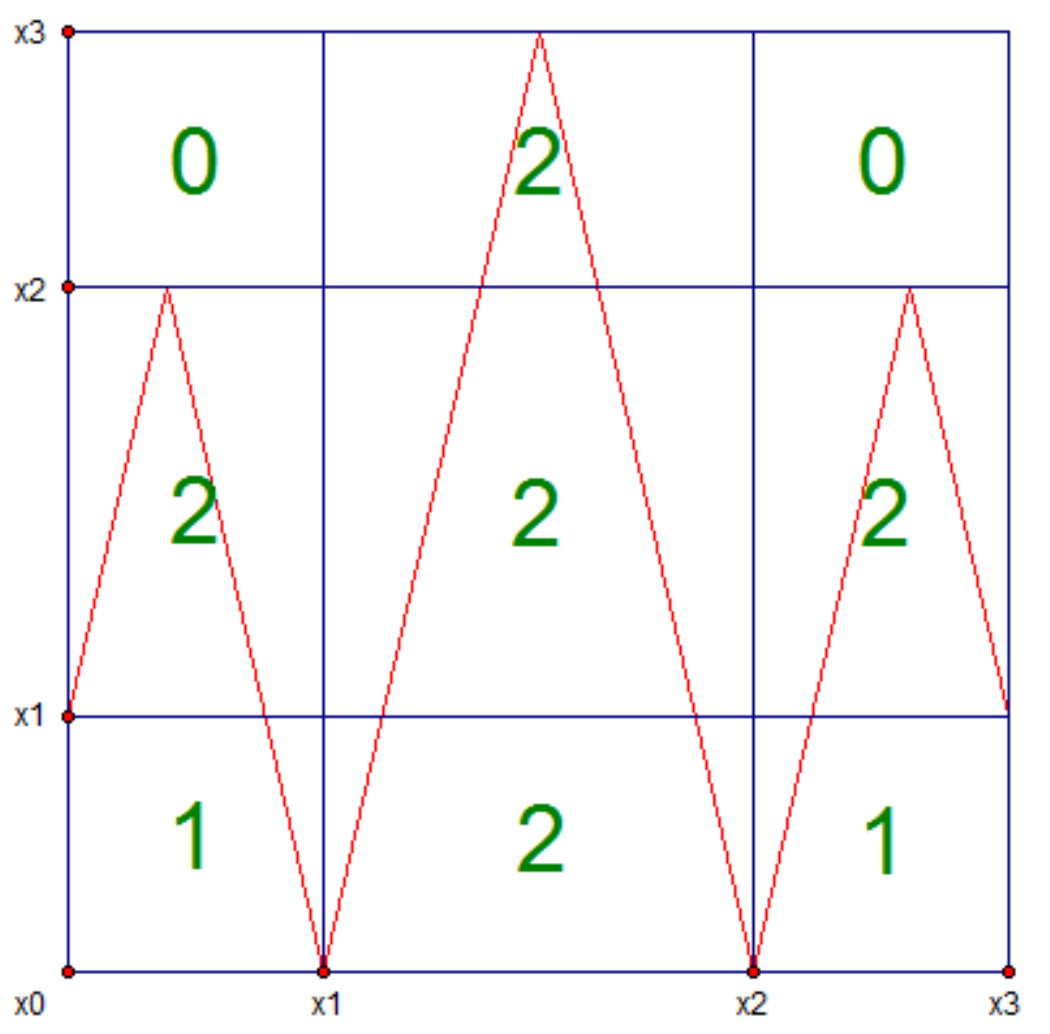}\includegraphics[width=0.4\textwidth]{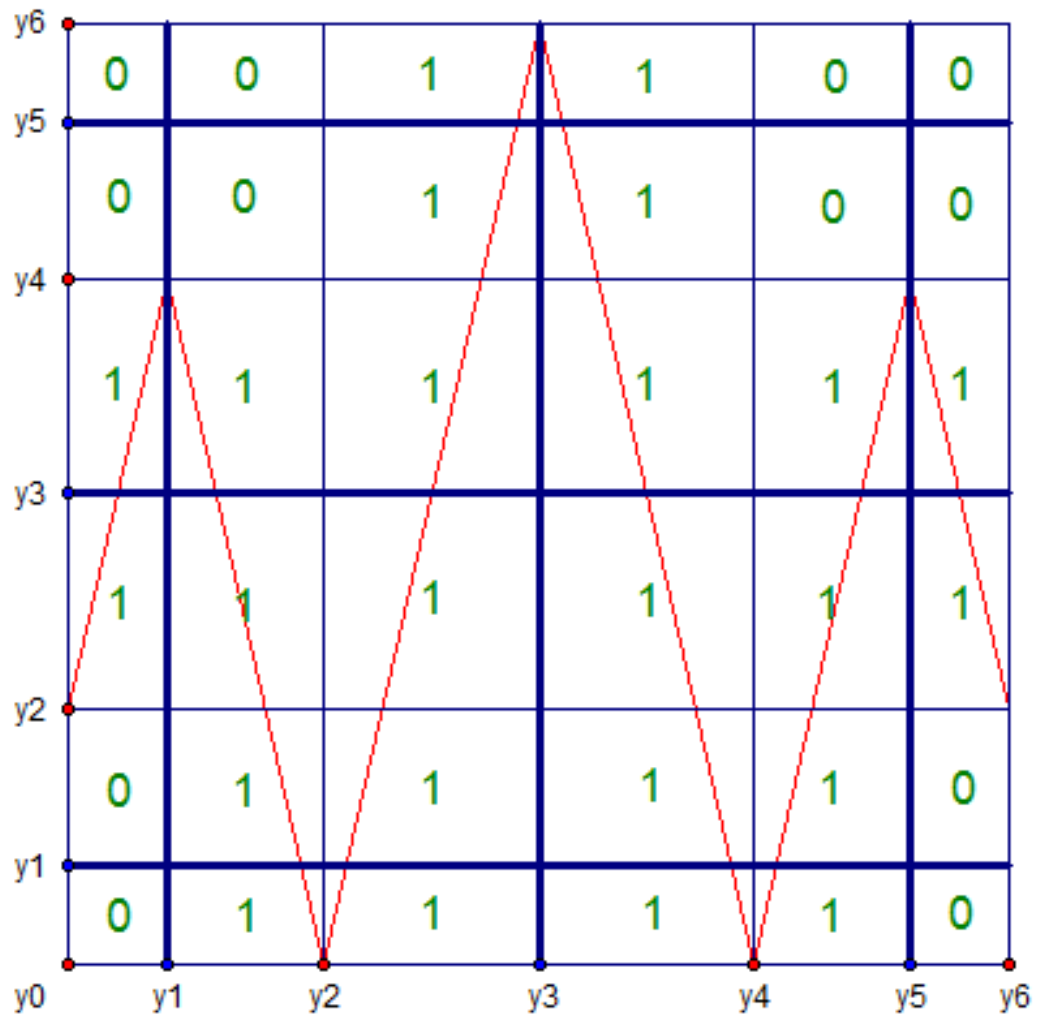}
\caption{\textbf{Left.} An odd-block matrix $M$ with a corresponding piecewise-linear
map $h_M$. \textbf{Right.} A matrix obtained by splitting the matrix on the left. } 
\label{fig:generaloddblock}
\end{figure}

Let $v$ be the $L^1$-normalized eigenvector of $M^T$ for the leading
eigenvalue. As we did in section 2, we get a partition $P=\lbrace x_0 , x_1, \dotsi , x_n \rbrace $ of $[0,1]$ so that $v_i=x_i-x_{i-1}$ for each $i=1,\dotsi,n$. Let's consider the $n\times n$ grid diagram on $[0,1]\times [0,1]$ generated by the partition $P$. Each $(i,j)$ box corresponds to $M_{ij}$ so that M is flipped upside down. Because $M$ is odd-block, it is always possible to draw a graph of piecewise-linear map $h_M$ so that the number of line segments of $h_M$ in each box is the same with the corresponding entry of $M$. As we draw on the grid diagram, the slopes of $h_M$ are either $\lambda$ or $-\lambda$. There may be few possible graphs with that property, but one can choose any of them. Note that the conversion depends on the choice of the graph.

Now we are in a position to convert $M$. Let $\overline{P}=\lbrace y_0
, y_1, \dotsi , y_n \rbrace $ be the union of $P$ and  the set of all
critical points of $h_M$. Then since the post-critical set of $h_M$ is
contained in $P$, $h_M(\overline{P})\subset \overline{P}$ and so
$\overline{P}$ is invariant under $h_M$. The extended incident matrix,
say $\overline{M}$, of $h_M$ associated with $\overline{P}$ is the
desired converted matrix. See the right part of Figure
\ref{fig:generaloddblock} for a resulting matrix $\overline{M}$
of this process. Note that $\overline{M}$ is inevitably singular because of duplicated
rows. 

Each entry of $\overline{M}$ is $0$ or $1$. Since $\overline{P}$
includes all critical points of $h_M$. The vector
$w=(w_1,\dotsi,w_m)$, $w_i=y_i-y_{i-1}$ for each $i=1,\dotsi,m$ is the
$L^1$-normalized eigenvector of $\overline{M}^T$ for $\lambda$. 
This is because the equation $(\overline{M}^Tw)_i=\lambda w_i$ simply
represents the length relation between $[y_{i-1},y_i]$ and
$h_M([y_{i-1},y_i])$. 

To prove $\overline{M}$ is aperiodic, we use following fact: 
For positive integer $p$, $i$-th column of $\overline{M}^p$ is
positive if and only if $h_M^p([y_{i-1},y_i])=[0,1]$. Since $M$ is
aperiodic and $h_M([y_{i-1},y_i])=[x_j, x_k]$ for some $j<k$,
$h_M^p([y_{i-1},y_i])$ eventually covers $[0,1]$ as p grows. Thus
$\overline{M}$ is aperiodic. Finally, from the Perron-Frobenius
theorem, we conclude that $\lambda$ is the leading eigenvalue of
$\overline{M}$ as it is the associated eigenvalue of the positive
eigenvector $w$ of the aperiodic matrix $\overline{M}$.

\vspace{.3in}

Department of Mathematical Sciences

KAIST

291 Daehak-ro, Yuseong-gu

Daejeon 34141, South Korea

E-mail: hrbaik@kaist.ac.kr

\vspace{.3in}

Department of Mathematics

Rutgers University

Hill Center - Busch Campus

110 Frelinghuysen Road

Piscataway, NJ 08854-8019, USA

E-mail: cwu@math.rutgers.edu

\vspace{.3in}

\textit{KyeongRo Kim:}

Department of Mathematical Sciences

KAIST

291 Daehak-ro, Yuseong-gu

Daejeon 34141, South Korea

E-mail: cantor14@kaist.ac.kr

\vspace{.3in}

\textit{TaeHyouk Jo:}

Department of Mathematical Sciences

KAIST

291 Daehak-ro, Yuseong-gu

Daejeon 34141, South Korea

E-mail: lyra95@kaist.ac.kr

\end{document}